\documentclass[10pt]{amsart}
\usepackage{mathrsfs}
\usepackage{amssymb,amsmath,amsfonts}
\usepackage[all,poly]{xy}
\usepackage{multirow}

\allowdisplaybreaks

\begin{document}
\theoremstyle{plain}
\newtheorem{thm}{Theorem}[section]
\newtheorem{theorem}[thm]{Theorem}
\newtheorem{lemma}[thm]{Lemma}
\newtheorem{corollary}[thm]{Corollary}
\newtheorem{proposition}[thm]{Proposition}
\newtheorem{conjecture}[thm]{Conjecture}
\newtheorem{obs}[thm]{}
\theoremstyle{definition}
\newtheorem{construction}[thm]{Construction}
\newtheorem{notations}[thm]{Notations}
\newtheorem{question}[thm]{Question}
\newtheorem{problem}[thm]{Problem}
\newtheorem{remark}[thm]{Remark}
\newtheorem{remarks}[thm]{Remarks}
\newtheorem{definition}[thm]{Definition}
\newtheorem{claim}[thm]{Claim}
\newtheorem{assumption}[thm]{Assumption}
\newtheorem{assumptions}[thm]{Assumptions}
\newtheorem{properties}[thm]{Properties}
\newtheorem{example}[thm]{Example}
\newtheorem{comments}[thm]{Comments}
\newtheorem{blank}[thm]{}
\newtheorem{defn-thm}[thm]{Definition-Theorem}

\def\diag{\operatorname{diag}}
\def\vol{\operatorname{vol}}
\makeatletter
 \def\imod#1{\allowbreak\mkern10mu({\operator@font mod}\,\,#1)}
 \makeatother


\title[An explicit formula of hitting times for random walks on graphs]
{An explicit formula of hitting times for random walks on graphs}

        \author{Hao Xu}
        \address{Center of Mathematical Sciences, Zhejiang University, Hangzhou, Zhejiang 310027, China;
        Department of Mathematics, University of Pittsburgh, 301 Thackery Hall, Pittsburgh, PA 15260, USA}
        \email{mathxuhao@gmail.com}

        \author{Shing-Tung Yau}
        \address{Department of Mathematics, Harvard University, Cambridge, MA 02138, USA}
        \email{yau@math.harvard.edu}

        \begin{abstract} We prove an explicit formula of
        hitting times in terms of enumerations of spanning trees for random walks on general connected graphs. We apply
        the formula to improve Lawler's bound of hitting times for general graphs, prove a sharp bound of hitting times
        for adjacent vertices and derive closed
        formulas of hitting times for some special graphs.
        \end{abstract}


    \maketitle

\section{Introduction}

Unless otherwise specified, throughout the paper, we assume $G=(V, E)$ to be an undirected graph
with $n=|V|$ vertices and without multi-edges or loops. The \emph{volume} of $G$ is
$\vol(G)=\sum_{v\in V} d_v$, where $d_v$ is the degree of $v$. Let $\tau(G)$ be the number of spanning trees of
$G$.
The \emph{Laplacian} of $G$ is the matrix $L=D-A$, where $D$ is
the diagonal matrix whose entries are the degree of the vertices and
$A$ is the adjacency matrix of $G$. For $x,y\in V$, $x\sim y$ denotes that they are adjacent vertices.

If $G$ is connected, eigenvalues of \emph{Chung's normalized Laplacian} $\mathcal{L}=D^{-1/2}LD^{-1/2}$
can be labeled by
$0=\lambda_1<\lambda_2\leq\lambda_3\leq\cdots\leq\lambda_{n}$ with
the corresponding orthonormal basis of eigenvectors
$v_1,v_2,\dots,v_{n}$. Let $v_i=(v_{i1},\dots,v_{in})^t$. Obviously $v_1(x)=\sqrt{d_x/\vol(G)},\forall x\in V$.

A \emph{random walk} on $G$ is a time-reversible finite Markov chain
that begins at some vertex, and at each step moves to a neighbor of the
present vertex $x$ with probability $1/d_x$.
The \emph{hitting time} $H(x,y)$ is
the expected number of steps to reach vertex $y$, when started from
vertex $x$. An excellent comprehensive survey of random walks on graphs can be
found in \cite{Lov}.

Chung and Yau \cite{CY} (see also \cite{Chu}) proved an explicit formula
of $H(x,y)$ in terms of the discrete Green function
\begin{equation}\label{eqgreen}
H(x,y)=\vol(G)\left(\frac{\mathscr G(y,y)}{d_y}-\frac{\mathscr G(x,y)}{\sqrt{d_x d_y}}\right).
\end{equation}
The discrete Green function $\mathscr{G}$ is uniquely defined by the equations
\begin{equation*}
 \mathscr{G}\mathcal{L} = \mathcal{L}\mathscr{G}  =
    I - P_0, \qquad
 \mathscr{G}P_0 = 0,\qquad P_0=v_1 v_1^t.
\end{equation*}
Chung-Yau's formula \eqref{eqgreen} is the starting point of our work \cite{XY}, which is continued here.

Lov\'asz proved a remarkable formula \cite[Thm. 3.1]{Lov} connecting hitting times to spectra of $\mathcal L$.
\begin{equation}\label{eqlov}
H(x,y)=\vol(G)\sum_{k=2}^n\frac{1}{\lambda_k}\left(\frac{v_{ky}}{d_y}-\frac{v_{kx}v_{ky}}{\sqrt{d_x d_y}}\right).
\end{equation}

Consider the graph $G$ as an electrical network, where each edge has unit resistance.
Tetali's electrical formula \cite{Tet} provides a powerful approach to
the computation of hitting times.
\begin{equation}\label{eqtet}
H(x,y)=\frac12 \sum_{z\in V(G)}d_z(R_{xy}+R_{yz}-R_{xz}),
\end{equation}
where $R_{xy}$ is the effective resistance between $x$ and $y$.

The paper is organized as follows: In \S
\ref{sectioninvariants}, we briefly review our previous work and prove an explicit formula of $H(x,y)$ in Theorem 2.7
together with some interesting applications.
In \S
\ref{secnetwork}, we present a proof of Tetali's electrical formula. In \S \ref{secexample}, we apply our formula to
recover some identities of hitting times of random walks on
lollipop graphs and unicycle graphs.

\

\noindent{\bf Acknowledgements} We thank Stephan Wagner for helpful comments on an earlier version of this paper.
We also thank a referee for providing an alternative proof of Corollary \ref{ht} (see Remark \ref{rm2}).

\vskip 30pt
\section{Explicit formulas of Hitting times}\label{sectioninvariants}

A \emph{vertex-weighted graph} is a graph $G$ together with a weight function $w: V(G)\rightarrow \mathbb R$. In
our case, $w_x$ at $x\in V(G)$ will usually be the
degree of $x$ in some ambient graph of $G$. So we may assume
$d_x\leq w_x\in \mathbb Z$. Denote by $d_G$ the weight function that takes $d_x$ for each $x\in V(G)$.

In \cite{XY}, we defined two invariants $R(G , w)$ and $Z(G , w)$ for a vertex-weighted graph $(G,
w)$. For the empty graph $\emptyset$, we define $R(\emptyset, w)=1$ and
$Z(\emptyset, w)=0$. For any given vertex $x\in V(G)$, they satisfy the recursive formulas
\begin{align} \label{eqR}
R(G, w)&=w_{x}R(G-\{x\}, w)-\sum_{{y\in V(G)}\atop {y\sim
x}}\sum_{P\in \mathscr P_{G}(x,y)} R(G-\{P\}, w),
\\\label{eqZ} Z(G, w)
&=w_{x}Z(G-\{x\}, w)-\sum_{{y\in V(G)}\atop {y\sim
x}}\sum_{P\in \mathscr P_{G}(x,y)} Z(G-\{P\}, w)\\
&\qquad +w_x^2R(G-\{x\},w)+\sum_{{u,v\in V(G)}\atop{u\neq
v}}\sum_{{{P_1 \in \mathscr P_G (x,u)}\atop{P_2 \in \mathscr
P_G(x,v)}}\atop P_1\cap P_2=x}w_u w_v R(G-\{P_1, P_2\}, w).\nonumber
\end{align}
where $\mathscr P_G(x,y)$ is the set of all simple undirected paths (with no repeated vertices) connecting
$x$ and $y$ in $G$. By convention $\mathscr P_G(x,x)$
consists of the trivial path $\{x\}$ only. Here $(G-\{P\},w)$ means the restriction of $w$ to the subgraph $G-\{P\}$.
Note that \eqref{eqR} and \eqref{eqZ} uniquely determine these invariants.

The invariants $R(G, w)$ and $Z(G, w)$ enjoy the following nice properties.
\begin{lemma}[\cite{XY}]
If $G$ has $k$ connected components $G_1, \dots,
G_k$, then
\begin{equation}
R(G, w)=\prod^k_{i=1}R(G_i, w),\qquad Z(G, w)=\sum^k_{i=1}Z(G_i, w)\prod^k_{{j=1}\atop {j\neq i}}
R(G_j, w).
\end{equation}
\end{lemma}

\begin{lemma}[\cite{XY}] \label{Z} We have
\begin{equation}\label{eqZ2}
Z(G,w)=\sum_{x,y\in V(G)}\sum_{P \in \mathscr P_G
(x,y)}w_x w_y R(G-\{P\},w).
\end{equation}
\end{lemma}
\begin{lemma}[\cite{XY}] \label{RZ} Let
$G$ be a connected graph, then $R(G,d_G)=0$ and $Z(G,d_{G})=\vol(G)^2\tau(G)$. For any $x,y\in V(G)$, we have
\begin{equation} \label{eqgreen13}
R(G-\{x\},d_G)=\sum_{P \in \mathscr P_G (x,y)}
R(G-\{P\},d_G)=\tau(G).
\end{equation}
\end{lemma}

\begin{remark}\label{rm1}
Recall a well-known result from linear algebra:
Let $G$ be a connected graph with possibly multi-edges but no loops. Let $L$ be its Laplacian matrix and $L'$
the matrix obtained by deleting the first row and column from $L$. Then
$$\tau(G) = \det(L').$$
Given a vertex-weighted graph $(G,w)$, define the completion graph $\overline G$ of
$G$ to be a multi-graph with $V(\overline G)=V(G)\cup\{\bullet\}$ and $E(\overline G)$ consists of $E(G)$ plus $w_v-d_v$
newly added edges between $\bullet$
and $v\in V(G)$ for each $v\in V(G)$.
It is not difficult to see that
$R(G,w)=\tau(\overline G)$. See the proof of \cite[Lem. 2.13]{XY} for details.
\end{remark}

The main result of \cite{XY} is an explicit formula of hitting times in terms of the invariants $R(G, w)$ and $Z(G, w)$.
\begin{theorem}[\cite{XY}]\label{hittingtime}
Let $G$ be a connected graph and $x, y\in V(G)$. Then
\begin{multline}\label{eqwalk}
H(x,y)=\frac{1}{\vol(G)\tau(G)}\Bigg(Z(G-\{y\},d_G)
-\sum_{P\in \mathscr P_{G}(x,y)}Z(G-\{P\},d_G)
\\\left. +\sum_{{u,v\in V(G)}\atop{u\neq v}}\sum_{{{P_1 \in
\mathscr P_G (x,u)}\atop{P_2 \in \mathscr P_G(y,v)}}\atop
P_1\cap P_2=\emptyset}d_u d_v R(G-\{P_1, P_2\},d_G)\right).
\end{multline}
\end{theorem}

\begin{remark}
According to \cite{Geo}, a graph $G$ is called \emph{reversible}
if $H(x,y)=H(y,x)$ holds for any $x,y\in V(G)$. It is not difficult to see that
\eqref{eqwalk} implies that $G$ is reversible if and only if $Z(G-\{x\},d_G)$ is
independent of the vertex $x$. An immediate corollary is that vertex-transitive graphs are reversible.
See \cite[Cor. 2.6]{Lov} for an alternative proof of this assertion.
It is interesting to compare with the result
in \cite[Cor. 4]{Tet} (also cf. \cite{GW}) that $G$ is reversible if and only if $\sum_{u\in V(G)}d_uR_{vu}$ is
independent of the vertex $v$, where $R_{vu}$
is the effective resistance between $v$ and $u$.
We hope our work will be useful to study the interesting problems on reversible graphs posed by Georgakopoulos \cite{Geo}.
\end{remark}

\begin{theorem}\label{hittingtime2}
Let $G$ be a connected graph and $x, y\in V(G)$. Then
\begin{equation}\label{eqhit}
H(x,y)=\frac{1}{\tau(G)}\sum_{u\in V(G)}d_u \sum_{P \in \mathscr P_G (x,u)\atop y\notin P}R(G-\{P,y\}, d_G).
\end{equation}
In fact, $R(G-\{P,y\}, d_G)=\tau(G/\{P,y\})$.
\end{theorem}
\begin{proof}By Lemma \ref{Z} and Remark \ref{rm1}, the formula \eqref{eqwalk} is equivalent to
\begin{multline}\label{eqwalk2}
H(x,y)=\frac{1}{\vol(G)\tau(G)}\sum_{u,v\in V(G)}d_u d_v\left(\sum_{P\in \mathscr P_{G}(u,v)\atop y\notin P} \tau(G/\{P,y\})\right.
\\\left. -\sum_{{P_1\in \mathscr P_{G}(x,y)\atop P_2\in \mathscr P_{G}(u,v)}\atop P_1\cap P_2=\emptyset} \tau(G/\{P_1,P_2\})
+\sum_{{{P_1 \in
\mathscr P_G (x,u)}\atop{P_2 \in \mathscr P_G(y,v)}}\atop
P_1\cap P_2=\emptyset} \tau(G/\{P_1, P_2\})\right),
\end{multline}
where $G/\{P,y\}$ and $G/\{P_1,P_2\}$ denote (multi-)graphs obtained from $G$ by contracting $\{P,y\}$ and $\{P_1,P_2\}$ to a point respectively.

Denote by $F(x,y,u,v)$ the bracket term of \eqref{eqwalk2}. We will show that for any fixed vertices $x,y,u$, $F(x,y,u,v)$ is independent of $v$.
This is obvious when $x=y$ or $u=y$, which forces $F(x,y,u,v)=0,\forall v\in V(G)$.
When $v=x$ or $y$, we have
\begin{equation}\label{eqt}
F(x,y,u,x)=F(x,y,u,y)=\sum_{P\in \mathscr P_{G}(x,u)\atop y\notin P} \tau(G/\{P,y\}).
\end{equation}
Assume $v\neq x,y$, we modify $G$ by adding an edge $uy$ if $u,y$ are not adjacent, namely we
define a simple graph $G'$ by
\begin{equation*}
G'= \left\{ \begin{array}{ll}
        G & \mbox{if $u\sim y$},        \\
        G\cup\{uy\} & \mbox{otherwise}.  \end{array}\right.
\end{equation*}
Denote by $\Omega(G')$ the set of spanning trees of $G'$.
Each of the three summations in $F(x,y,u,v)$ counts a subset of $\Omega(G')$. They are respectively equal to
\begin{gather}\label{eqt2}
\#\{T\in\Omega(G')\mid T \mbox{ contains $uy$ and a path from $u$ to $v$ not containing $y$} \},\\ \label{eqt3}
\#\{T\in\Omega(G')\mid T \mbox{ contains $uy$, a path $P_1$ from $x$ to $y$ and a path $P_2$ from $u$ to $v$}\\\nonumber
 \qquad\qquad\mbox {such that } P_1\cap P_2=\emptyset\},\\\label{eqt4}
\#\{T\in\Omega(G')\mid T \mbox{ contains $uy$, a path $P_1$ from $u$ to $x$ and a path $P_2$ from $v$ to $y$}\\\nonumber
\qquad\qquad\mbox {such that } P_1\cap P_2=\emptyset\}.
\end{gather}
It is not difficult to see that
\begin{align*}
F(x,y,u,v)&=\eqref{eqt2}-\eqref{eqt3}+\eqref{eqt4}\\
&=\#\{T\in\Omega(G')\mid T \mbox{ contains $uy$ and a path from $u$ to $x$} \}\\
&=F(x,y,u,x),
\end{align*}
which proves that $F(x,y,u,v)$ is independent of $v\in V(G)$. The last equation used \eqref{eqt}.
Since $\tau(G/\{P,y\})=R(G-\{P,y\}, d_G)$, we get \eqref{eqhit} immediately from \eqref{eqwalk2}.
\end{proof}

As an application of the above theorem, we give a simple proof of the well-known inequality $H(x,y)\leq O(n^3)$,
where $n=|V(G)|$.
\begin{corollary}
Let $G$ be a connected graph with $n$ vertices and $x, y\in V(G)$. Then
\begin{equation}\label{eqt5}
H(x,y)\leq (n-1)^3.
\end{equation}
If in addition $d_x\leq k, \forall x\in V(G)$, then
\begin{equation}\label{eqt6}
H(x,y)\leq k(n-1)^2.
\end{equation}
\end{corollary}
\begin{proof}
Fix $x,y,u\in V(G)$ with $y\neq u$.
Given a spanning tree $T\in \Omega(G)$ and an edge $e\in E(T)$, denote by $T(e)$ a subgraph of $G'$
obtained from $T$ by removing $e$ and adding an edge $uy$ if $uy\notin E(T)$, namely
\begin{equation*}
T(e)= \left\{ \begin{array}{ll}
        T & \mbox{if $uy\in T$},        \\
        \{T-e\}\cup\{uy\} & \mbox{if $uy\notin T$}.  \end{array}\right.
\end{equation*}

Define a subset $S$ of $\Omega(G)\times E(G)$ by
$$S=\{(T,e)\mid T\in\Omega(G),e\in E(T),T(e)\in\Omega(G')\}$$
and $S'=\{T\in\Omega(G')\mid T\mbox{ contains $uy$}\}$.
Then the map $(T,e)\rightarrow T(e)$ is a surjective map from $S$ to $S'$.
Therefore we have
\begin{align*}
\sum_{P\in \mathscr P_{G}(x,u)\atop y\notin P} \tau(G/\{P,y\})&=\#\{T\in\Omega(G')\mid T \mbox{ contains $uy$ and a path from $u$ to $x$} \}\\
&\leq |S'|\leq |S|\leq (n-1)\tau(G).
\end{align*}
Let $d_{\max}=\max\{d_v\mid v\in V(G)\}$. Then from \eqref{eqhit}, we have
\begin{equation*}
H(x,y)\leq d_{\max}(n-1)^2,
\end{equation*}
which implies \eqref{eqt5} and \eqref{eqt6}.
\end{proof}

\begin{remark}
An $O(n^3)$ upper bound for hitting times was first proved by Aleliunas et al. \cite{AKL}.
Inequalities \eqref{eqt5} and \eqref{eqt6} with slightly weaker bounds $n(n-1)^2$ and $kn(n-1)$ respectively were obtained by
Lawler \cite{Law}. A sharp bound of $H(x,y)$ with leading term $(4/27)n^3$ was obtained by Brightwell and Winkler \cite{BW}, who also showed
that lollipop graphs maximize $H(x,y)$.
\end{remark}

\begin{corollary}
Let $G$ be a connected graph with $m$ edges and $xy\in E(G)$. Then
\begin{equation}\label{eqt10}
H(x,y)\leq 2m-d_y.
\end{equation}
\end{corollary}
\begin{proof}
For any $u\in V(G)$ with $u\neq y$, it is not difficult to see that
\begin{align*}
\sum_{P\in \mathscr P_{G}(x,u)\atop y\notin P} \tau(G/\{P,y\})&=\#\{T\in\Omega(G)\mid T \mbox{ contains $xy$ and a path from $u$ to $x$} \}\\
&\leq \tau(G).
\end{align*}
Thus \eqref{eqhit} implies that
$$H(x,y)\leq \frac{1}{\tau(G)}\sum_{u\in V(G)\atop u\neq y}d_u \tau(G)=2m-d_y,$$
as claimed.
\end{proof}

\begin{remark}
It is well-known (cf. \cite[p.8]{Lov}) that the commute time $\kappa(x,y):=H(x,y)+H(y,x)\leq 2m$ whenever $xy\in E(G)$. The inequality \eqref{eqt10} seems new and is sharp for the path graph,
where $x$ and $y$ are respectively the next-to-right endpoint and the right endpoint.

Let $\mathscr S=\{u\in V(G)\mid \mbox{There is a path from $x$ to $u$ not passing through $y$} \}$.
If $xy\in E(G)$ is a cut edge of $G$, then it was proved in \cite{AKL,Fei} that $\kappa(x,y)=2m$, or equivalently $H(x,y)=2|E(G')|-1$,
where $G'$ is the subgraph obtained by removing all vertices in $V(G)/\{\mathscr S\cup y\}$ from $G$.
The latter equality can also be proved easily using \eqref{eqhit}. First note that $H(x,y)$ is the same for random walks on either $G$ and $G'$.
Moreover, for each spanning tree $T$ of
$G'$ and $u\in\mathscr S$, there exists a path from $x$ to $u$. Therefore,
$$H(x,y)= \frac{1}{\tau(G')}\sum_{u\in\mathscr S}d_u \tau(G')=2|E(G')|-1,$$
where the last equation follows from the fact that the degree of $y$ in $G'$ is equal to $1$.

\end{remark}

\begin{corollary}\label{ht}
Let $G$ be a connected graph on $n$ vertices. If there is a vertex $y$ with degree $n-1$, then for any $x\in V(G)$ we have
\begin{equation}\label{eqt11}
H(x,y)\leq \max\{d_u\mid u\in \mathscr S \},
\end{equation}
 where $\mathscr S=\{u\in V(G)\mid \mbox{There is a path from $x$ to $u$ not passing through $y$} \}$.
\end{corollary}
\begin{proof}
Fix any $x\in V(G)$ with $x\neq y$, we define
$\Omega_{xy}=\{T\in\Omega(G)\mid xy\in T\}$ and
$$V_T=\{u\in V(G)\mid \mbox{$T$ contains a path from $x$ to $u$ not passing through $y$} \}.$$
Let $S=\{(T,u)\mid T\in\Omega_{xy},u\in V_T\}$. Define a map $f:S\rightarrow\Omega(G)$ by
\begin{equation*}
f(T,u)= \left\{ \begin{array}{ll}
        T & \mbox{if $u=x$},        \\
        \{T-xy\}\cup\{uy\} & \mbox{if $u\neq x$},  \end{array}\right.
\end{equation*}
where we used the fact that $d_y=n-1$.
It is not difficult to see that $f$ is injective. Thus we have
\begin{equation*}
\sum_{u\in V(G)}\sum_{P\in \mathscr P_{G}(x,u)\atop y\notin P} \tau(G/\{P,y\})=|S|\leq \tau(G).
\end{equation*}
Therefore \eqref{eqhit} implies \eqref{eqt11}.
\end{proof}

\begin{remark} Without loss of generality, we may assume $\mathscr S=V(G)-\{y\}$ in the above corollary.
Eq. \eqref{eqt11} refines the result of Palacios \cite[Thm. 3.1]{Pal}, who proved $H(x,y)\leq n-1$ under the same
condition of the above corollary by using inequalities between matrix norms.
\end{remark}

\begin{remark}\label{rm2}
A simple probabilistic proof of Corollary \ref{ht} was provided by a referee:
Since the probability of moving to y is
at least $p = 1/\max\{d_u\mid u \in \mathscr S\}$ in every step, the expected hitting time is at most
the expected value of a geometric random variable with parameter $p$, which is $1/p$.
\end{remark}

\begin{remark}
Spanning trees have been extensively used to estimate hitting times \cite{AKL,CP,Fei}. Some of their
arguments are very technical. It shall be interesting to see how to
apply \eqref{eqhit} to recover their estimates of hitting times.
Explicit formulas of hitting times valid on general graphs are very rare.
As shown in \cite[\S 4]{XY}, Eq. \eqref{eqwalk} is very useful in studying hitting times on general graphs.
We will show in \S \ref{secexample} that Eq. \eqref{eqhit} is very efficient in getting closed formulas for hitting times
on graphs with few cycles.
\end{remark}

\vskip 30pt
\section{Random walks and electric networks}\label{secnetwork}

There has been a large amount of work on connections between electrical networks and random walks
on graphs. Chandra et al. \cite{CRR} proved that the commute time $\kappa(x,y)$ can
be expressed in terms of the effective resistance $\kappa(x,y)=\vol(G)R_{xy}$.
Tetali's electrical formula \cite{Tet}, expressing $H(x,y)$ in terms of the effective resistance,
was originally proved by using the reciprocity theorem of electrical networks. It was used to prove, among others,
closed formulas of hitting times for trees and unicycle graphs \cite{CZ}.
As an illustration of the effectiveness of \eqref{eqwalk}, we use it to prove Tetali's formula.
\begin{theorem}[Tetali \cite{Tet}] On a connected graph $G$,
\begin{equation}\label{eqtet2}
H(i,j)=\frac12\left(\kappa(i,j)+\sum_{q\in V(G)}\frac{d_q}{\vol(G)}[\kappa(q,j)-\kappa(q,i)]\right).
\end{equation}
\end{theorem}
\begin{proof}
By \eqref{eqwalk},
\begin{align}\label{eqt7}
\vol(G)[2H(i,j)-\kappa(j,i)]&=\vol(G)[H(i,j)-H(j,i)]\\ \nonumber
&=\frac{1}{\tau(G)}\bigg(Z(G-\{j\},d_G)-Z(G-\{i\},d_G)\bigg).
\end{align}
and by \eqref{eqwalk} and \eqref{eqZ2},
\begin{multline}\label{eqt8}
\sum_{q\in V(G)} d_q[H(q,j)+H(j,q)]\\
=\frac{1}{\vol(G)\tau(G)}\left(\sum_{q\in V(G)} d_q[Z(G-\{q\},d_G)+Z(G-\{j\},d_G)]
-2\sum_{q\in V(G)}d_q \sum_{P\in \mathscr P_{G}(q,j)}Z(G-\{P\}, d_G)\right.\\
\left.+2\sum_{q\in V(G)}d_q\sum_{u,v\in V(G)}\sum_{{{P_1 \in
\mathscr P_G (q,u)}\atop{P_2 \in \mathscr P_G(j,v)}}\atop
P_1\cap P_2=\emptyset}d_u d_v R(G-\{P_1, P_2\},
d_G)\right)\\
=\frac{1}{\vol(G)\tau(G)}\sum_{q\in V(G)} d_q[Z(G-\{q\},d_G)+Z(G-\{j\},d_G)]\\
-\frac{1}{\vol(G)\tau(G)}\left(2\sum_{q,u,v\in V(G)}\sum_{{{P_1 \in
\mathscr P_G (q,j)}\atop{P_2 \in \mathscr P_G(u,v)}}\atop
P_1\cap P_2=\emptyset}d_q d_u d_v R(G-\{P_1, P_2\},
d_G)\right.\\
\left.-2\sum_{q,u,v\in V(G)}\sum_{{{P_1 \in
\mathscr P_G (q,u)}\atop{P_2 \in \mathscr P_G(j,v)}}\atop
P_1\cap P_2=\emptyset}d_q d_u d_v R(G-\{P_1, P_2\},
d_G)\right)\\
=\frac{1}{\vol(G)\tau(G)}\sum_{q\in V(G)} d_q[Z(G-\{q\},d_G)+Z(G-\{j\},d_G)].
\end{multline}
The vanishing of the bracket term in the last equation can be seen by switching $q$ and $v$.
By \eqref{eqt8}, we get
\begin{multline}\label{eqt9}
\sum_{q\in V(G)}d_q [\kappa(q,j)-\kappa(q,i)]=\sum_{q\in V(G)} d_q[H(q,j)+H(j,q)-H(q,i)-H(i,q)]\\
=\frac{1}{\vol(G)\tau(G)}\sum_{q\in V(G)} d_q\bigg(Z(G-\{j\},d_G)-Z(G-\{i\},d_G)\bigg),
\end{multline}
which, together with \eqref{eqt7}, implies Tetali's formula \eqref{eqtet2}.
\end{proof}

Tetali's formula \eqref{eqtet2} also gives an expression of hitting times in terms
of numbers of spanning trees via the following equation (cf. \cite{Lov})
\begin{equation}
\kappa(x,y)=\vol(G)\frac{\tau(G')}{\tau(G)},
\end{equation}
where $x\neq y\in V(G)$ and $G'$ is the
graph obtained from $G$ by identifying $x$ and $y$. In contrast to Tetali's formula, the formula \eqref{eqhit}
does not have negative terms and thus is more efficient for bounding hitting times.

\vskip 30pt
\section{Two examples}\label{secexample}

The weight function $w$ may be written as a sequence
$[w_1,\dots,w_n]$ with $n=|V(G)|$ once we specify a natural labeling of $V(G)$.
The following lemmas will be used in Example \ref{ex1}.
\begin{lemma}[{\cite[\S 3]{XY}}] \label{R1}
Let $P_n$ and $K_n$ be the path and the complete graph on $n$ vertices respectively. Then
\begin{align}
R(P_n, [2^n])&=n+1,\\
k_{n,m}:=R(K_n, [m^n])&=(m-n+1)(m+1)^{n-1}. \label{eqcomplete}
\end{align}
\end{lemma}

\begin{lemma}\label{R2}
Let $L_{m,n}$ be a lollipop graph obtained by attaching a path $P_n$ to $K_m$.
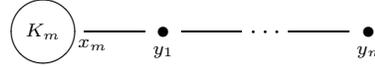
\begin{equation*}
\xymatrix@C=8mm{ *+++[o][F-]{\scriptstyle K_m} \ar@{-}[r]_<<{x_m} &\underset{y_1}{\bullet} \ar@{-}[r] & \cdots \ar@{-}[r] & \underset{y_n}{\bullet}}
\end{equation*}
More precisely $V(G)=\{x_1,\dots,x_m,y_1,\dots,y_n\}$
and $(x_i,x_j)\in E(G)$ for $1\leq i<j\leq m$; $(x_m,y_1)\in E(G)$;
$(y_i,y_i+1)\in E(G)$ for $1\leq i<n$. Define a weight function $D_k$ on
$L_{m,n}$ by
\begin{equation*}
 D_k(v) = \left\{ \begin{array}{ll}
        d_v+k & \mbox{if $v=x_j,\,1\leq j\leq m$},  \\
        2 & \mbox{if $v=y_j,\,1\leq j\leq n$.}
        \end{array}\right.
\end{equation*}
Then $r_{m,n,k}:=R(L_{m,n}, D_k)$ is equal to
\begin{equation*}
r_{m,n,k}=((m+k)(n+1)-n)(k+1)(m+k)^{m-2}-(m-1)(m+k)^{m-2}(n+1).
\end{equation*}
\end{lemma}
\begin{proof}
Taking $x=x_m$ in \eqref{eqR}, we have
\begin{multline}
R(G_{m,n},D_k)=(m+k)R(K_{m-1},[(m-1+k)^{m-1}])\cdot
R(P_n,[2^n])\\ -R(K_{m-1},[(m-1+k)^{m-1}]\cdot
R(P_{n-1},[2^{n-1}])\\-(m-1)\sum^{m-2}_{i=0}\binom{m-2}{i}i!R(K_{m-2-i},(m-1+k)^{m-2-i})R(P_n,[2^n])\\
=((m+k)(n+1)-n)(k+1)(m+k)^{m-2}-(m-1)(m+k)^{m-2}(n+1),
\end{multline}
as claimed.
\end{proof}

\begin{example}\label{ex1}
Let $G$ be a lollipop graph $L_{N,N}$ with $N\geq2$.
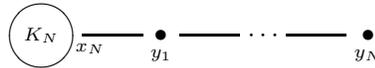
\begin{equation*}
\xymatrix@C=8mm{ *+++[o][F-]{\scriptstyle K_N} \ar@{-}[r]_<<{x_N} &\underset{y_1}{\bullet} \ar@{-}[r] & \cdots \ar@{-}[r] & \underset{y_N}{\bullet}}
\end{equation*}
By \eqref{eqhit}, we have
\begin{multline*}
H(x_1,y_N)=\frac{1}{N^{N-2}}\left((N-1)r_{N-1, N-1,
1}+(N-1)(N-2)\sum^{N-3}_{i=0}\binom{N-3}{i}i!r_{N-2-i, N-1, i+2}\right.\\
+(N-1)(N-2)\sum^{N-3}_{i=0}\binom{N-3}{i}(i+1)!R(K_{N-(i+3)},[(N-1)^{N-(i+3)}])R(P_{N-1},[2^{N-1}])\\
+N\sum^{N-2}_{i=0}\binom{N-2}{i}i!R(K_{N-(i+2)},[(N-1)^{N-(i+2)}])R(P_{N-1},[2^{N-1}])\\
\left.+2\sum^{N-2}_{i=0}\binom{N-2}{i}i!R(K_{N-(i+2)},[(N-1)^{N-(i+2)}])\sum^{N-1}_{j=1}R(P_{N-1-j},[2^{N-1-j}]\right).
\end{multline*}
The five terms in the bracket respectively correspond to (i) $u=x_1$; (ii) $u=x_j,2\leq j\leq N-1$ and $x_N\notin P\in\mathscr P_G(x_1,u)$;
(iii) $u=x_j,2\leq j\leq N-1$ and $x_N\in P\in\mathscr P_G(x_1,u)$; (iv) $u=x_N$; (v) $u=y_j,1\leq j\leq N-1$.

By Lemmas \ref{R1} and \ref{R2}, it is not difficult to get
\begin{multline}
H(x_1,y_N)=\frac{1}{N^{N-2}}\Bigg((N-1)r_{N-1, N-1,
1}\\
+(N-1)(N-2)\sum^{N-3}_{i=0}\frac{N-3}{(N-3-i)!}[r_{N-2-i, N-1, i+2}
+(i+1)k_{N-(i+3),(N-1)}\cdot
N]\\
+\sum^{N-2}_{i=0}\frac{(N-2)!}{(N-2-i)!}\bigg[k_{N-(i+2),(N-1)}\cdot N^2+2
k_{N-(i+2),N-1}\sum^{N-1}_{j=1}(N-j)\bigg]\Bigg)\\
=N^3+N-1.
\end{multline}
\end{example}

The above example was discussed in \cite{Law,Pal} by different approaches.

By using Tetali's electrical formula, Chen and Zhang \cite{CZ}
obtained an explicit formula for hitting times of random walks on unicycle graphs.
In the next example, we apply Theorem \ref{hittingtime2} to give
a more direct derivation of Chen-Zhang's formula.
\begin{example}
We follow the notations of \cite{CZ}.
Let $G$
be a connected unicyclic graph with a unique cycle $C$ of length $l$.
Let $V(C) = \{1, 2,\dots, l\}$ and $T_i,\,1\leq i\leq l$ the tree
component of $G\backslash E(C)$ containing $i$. Denote $m_i =|E(T_i)|$.
Given $i,j,k\in V(C)$, denote by $P_{ijk}$ the path from $i$ to
$k$ containing $j$. Let $m_{ij}$ and $m_{jk}$ be the lengths of subpaths of $P_{ijk}$ from $i$ to $j$ and $j$ to $k$
respectively.

First we assume that there are two distinct vertices $i,j\in V(C)$ such that
$a\in V(T_i), b\in V(T_j)$. Let $G_0$ be the subgraph of $G$ induced by $V(P_{ai})\cup V(P_{jb})\cup V(C)$ and $m_v$
the number of edges in the component of $G\backslash E(G_0)$ containing $v$. Define a map $f:V(G)\rightarrow V(G_0)$
by setting $f(u)=v$ if $u\in V(G)$
belongs to the component of $G\backslash E(G_0)$ containing $v\in V(G_0)$.

Considering the three cases $u\in f^{-1}(P_{ai})$, $u\in f^{-1}(P_{jb})$ and $u\in f^{-1}(V(C)\backslash\{i,j\})$ in \eqref{eqhit}, we get
\begin{multline*}
H(a,b)=\frac{1}{l}\sum_{u\in V(G)}d_u \sum_{P \in \mathscr P_G (a,u)\atop y\notin P}\tau(G/\{P,b\})\\
=\sum_{v\in P_{ai}}2m_v\left(d(v,i)+d(j,b)+\frac{d(i,j)(l-d(i,j))}{l}\right)
+\sum_{v\in P_{jb}}2m_v d(v,b)\\
+\sum_{k\in V(C)\backslash\{i,j\}}(2m_k+2)\left(d(j,b)+\frac{m_{ij}m_{jk}}{l}\right)\\
=2\sum_{v\in P_{ai}}m_v\left(d(v,i)+d(j,b)+\frac{d(i,j)(l-d(i,j))}{l}\right)
+2\sum_{v\in P_{jb}}m_v d(v,b)\\
+2\sum_{k\in V(C)\backslash\{i,j\}}m_k\left(d(j,b)+\frac{m_{ij}m_{jk}}{l}\right)\\
+d(a,i)^2+d(j,b)^2+2(l+d(a,i))d(j,b)+\frac{l+2d(a,i)}{l}d(i,j)(l-d(i,j)).
\end{multline*}

If $a,b\in V(T_i)$ for some $i$, then it is not difficult to show that
$$H(a,b)=d(a,b)^2+2\sum_{v\in V(P)}m_v d(v,b).$$

\end{example}

$$ \ \ \ \ $$

\end{document}